\documentclass{article}

\usepackage{amsfonts,amsmath,amssymb,amsthm}
\usepackage{amscd}
\usepackage{mathrsfs}
\usepackage[dvipsnames,table,xcdraw]{xcolor}
\usepackage[colorlinks = true,
            linkcolor = Fuchsia,
            urlcolor  = ForestGreen,
            citecolor = WildStrawberry,
            anchorcolor = blue]{hyperref}

\usepackage[all,2cell]{xy} \UseAllTwocells 
\usepackage{tikz}
\usetikzlibrary{cd}
\usetikzlibrary{arrows,automata}
\usetikzlibrary{decorations.markings}
\usetikzlibrary{decorations.pathmorphing,shapes}
\usepackage{enumitem}

\usepackage[english]{babel}

\usepackage[letterpaper,top=2cm,bottom=2cm,left=3cm,right=3cm,marginparwidth=1.75cm]{geometry}

\usepackage[style=numeric, backend=bibtex,url=false,maxnames=10]{biblatex}
\usepackage{csquotes}
\makeatletter
\AtEveryBibitem{%
  \ifnameundef{author}{}{%
    \iffieldequals{fullhash}{\bbx@lasthash}
      {\renewcommand{\printnames}[1]{\bibnamedash}}
      {\savefield{fullhash}{\bbx@lasthash}}%
  }%
}
\makeatother

\usepackage{amsmath}
\usepackage{graphicx}

\def\lra{\longrightarrow}

\def\Q{\mathbb Q}
\def\Z{\mathbb Z}

\newcommand{\id}{\mathsf{id}}
\newcommand{\Id}{\mathsf{Id}}

\newcommand{\Hom}{\mathsf{Hom}}

\newcommand{\Coh}{\mathsf{Coh}}
\newcommand{\slt}{\mathfrak{sl}(2)}
\newcommand{\mbGm}{\mathbb{G}_m}
\newcommand{\Sym}{\mathrm{Sym}}
\newcommand{\colim}{\mathrm{colim}}
\newcommand{\Int}{\mathrm{Int}}

\newcommand{\kk}{\mathbf{k}}

\newcommand{\mcO}{\mathcal{O}}
\newcommand{\mcL}{\mathcal{L}}
\newcommand{\mcC}{\mathcal{C}}
\newcommand{\mcS}{\mathcal{S}}
\newcommand{\mcF}{\mathcal{F}}
\newcommand{\mcA}{\mathcal{A}}

\newcommand{\mcU}{\mathcal{U}}

\newcommand{\mbs}{\mathbf{s}}
\newcommand{\wtR}{\widetilde{R}}
\newcommand{\wtK}{\widetilde{K}}
\newcommand{\PP}{\mathbb{P}}
\newcommand{\UzeroA}{\mcU^0_{\mcA}}
\newcommand{\UzeroAzero}{\mcU_{\mcA}^{\overline{0}}}
\newcommand{\UzeroAone}{\mcU_{\mcA}^{\overline{1}}}
\newcommand{\Rees}{\operatorname{Rees}^{(2)}_{\mathscr F}}

\newcommand{\bbinom}[2]{
  \begin{bmatrix}
    #1 \\
    #2
  \end{bmatrix}  
}

\newcommand{\bmat}[3]{
  \begin{bmatrix}
    #1 ; #2 \\
    #3
  \end{bmatrix}  
}

\newcommand{\oplusop}[1]{{\mathop{\oplus}\limits_{#1}}}

\theoremstyle{definition}
\newtheorem{thm}{Theorem}[section]
\newtheorem{lem}[thm]{Lemma}
\newtheorem{remark}[thm]{Remark}
\newtheorem{cor}[thm]{Corollary}
\newtheorem{prop}[thm]{Proposition}

\theoremstyle{definition}
\newtheorem{definition}[thm]{Definition}

\title{A categorification of the integral form of the Cartan subalgebra for quantum $\slt$}
\author{Mikhail Khovanov}

\date{August 15, 2026}

\begin{document}
\maketitle

\begin{abstract}
 We propose a categorification of the Lusztig integral form of the quantum Cartan subalgebra for $\slt$ via a  colimit of categories of equivariant coherent sheaves on finite-dimensional projective spaces. The nonequivariant version of our construction yields a nonsemisimple  abelian categorification of the ring of integer-valued polynomials. 
\end{abstract}

\tableofcontents

%%%%%%%%%%%%%%%%%%
%
% INTRODUCTION 
%
%%%%%%%%%%%%%%%%%%

\section{Introduction}

In a recent paper~\cite{KhCoherent}, a commutative algebra $R$ lifting the Lusztig integral form $\UzeroA$ of the Cartan subalgebra of quantum $\slt$ was categorified via the direct sum of categories of $\mbGm$-equivariant coherent sheaves on $\PP^n$, over all $n$. 

In the present paper, writing the torus $\mbGm$ as $T$, we form a suitable colimit $\Coh^T_{\PP}$ of the categories of coherent sheaves via exact functors 
\[
\Coh^{T}(\PP^n) \stackrel{\iota_n}{\lra} \Coh^{T}(\PP^{n+2}),
\]
which are pushforwards $\iota_n=(i_n)_{\ast}$ for (equivariant) inclusions 
\[
\PP^n \stackrel{i_n}{\lra} \PP^{n+2}. 
\] 
The colimit category $\Coh^{T}_{\PP}$ 
is abelian and carries a biexact convolution functor. The main result, given by Theorem~\ref{thm_abelianK0}, is the identification of the Grothendieck ring of  $\Coh^{T}_{\PP}$ with the Lusztig integral form of the Cartan. Lusztig elements $\bmat{K}{c}{n}$ lift to objects $\mathscr B_{c,n}$ of  $\Coh^{T}_{\PP}$, which are images of line bundles $\mcO(-c)$ on $\PP^n$'s, see \eqref{eq_Bscript_definition}. 
Guti\'errez-Mart\'inez-Szwej-Wildon relation in $\UzeroA$ given in~\cite[Theorem~1.2]{GMSW1}
is categorified by the direct sum decomposition of the pushforward in~\eqref{eq_Bscript_product}. A categorical lifting of comultiplication in $\UzeroA$ is provided as well. The problem of categorifying the integral form of the Cartan subalgebra for quantum $\slt$ was posed and studied in~\cite{MartinezRuiz2025}. 

In Section~\ref{sec_noneq} we consider the non-equivariant version of our construction and its relation to the ring of integer-valued polynomials and its Rees algebra. 

\vspace{0.07in}

 {\bf Acknowledgments.} 
 The author would like to acknowledge active discussions with ChatGPT 5.6 Sol while writing the paper. The author was partially supported by the Simons Collaboration Award 994328 ``New Structures in Low-Dimensional Topology''. 

%%%%%%%%%%%%%%%%%%%%%%%
%
% Introduction 
%
%%%%%%%%%%%%%%%%%%%%%%%

\section{Ring \texorpdfstring{$R$}{R} and the integral form of the Cartan subalgebra}

\subsection{Ring \texorpdfstring{$R$}{R}}
We recall some notations and results from~\cite{KhCoherent}. 
Let 
\[
\mcA \ := \ \Z[q,q^{-1}], 
\]
and let $R=\oplus_{n\ge 0}R_n$ be the commutative $\Z_+$-graded $\mcA$-algebra with generators $K_{c,n}$, over all $c\in \Z$, $n\in \Z_+$ and defining relations 
\begin{equation}\label{eq_q_binom}
    \sum_{k=0}^{n+1} (-1)^k \bbinom{n+1}{k} K_{c+k,n} =0, 
\end{equation}
for any $n\ge 0$ and $c\in \Z$, and 
\begin{equation}\label{eq_multiplication}
K_{c,n}K_{b,m} = \sum_{k \ge 0} \begin{bmatrix} n - c + b \\ k - c \end{bmatrix} \begin{bmatrix} m - b + c \\ k - b \end{bmatrix} 
K_{k,n+m}, 
\end{equation}
for $n,m\ge 0$ and $0\le c\le n$, $0\le b\le m$. The sum above is over $k$ such that $\max(b,c) \le k \le  \min(n+b, m+c)$. The grading is given by $\deg K_{c,n}=n$. We use balanced $q$-binomials:
\begin{equation}\label{eq_balanced}
\bbinom{n}{k}:=\frac{[n]!}{[k]![n-k]!}, \ \ [n]! :=[n][n-1]\dots [1], \ \ [n]=\frac{q^n-q^{-n}}{q-q^{-1}}\in \mcA=\Z[q,q^{-1}],  
\end{equation}
and the $q$-binomial equals $0$ if $k<0$ or $k>n$. 

Let $\mcS=\{(c,n)|0\le c\le n\}$. 
We write $K_{\mbs}, \mbs\in\mcS$ to denote generators from this set. Recall from~\cite{KhCoherent} that each graded piece $R_n$ is a free $\mcA$-module with a basis 
$(K_{0,n}, K_{1,n}, \dots, K_{n,n})$, and that $K_{\mbs}$, over all $\mbs\in \mcS$, is a basis of the free $\mcA$-module $R$. 

$R_1$ is spanned by
\begin{equation}\label{eq_A01}
A_0:=K_{0,1}, \ \   A_1:=K_{1,1}.
\end{equation}
Equation \eqref{eq_q_binom} for $n=1$ specializes to 3-term relations $K_{c+2,1}= [2]K_{c+1,1}- K_{c,1}$, which imply 
\begin{equation}\label{eq_Kc1}
K_{c,1}=[c]A_1-[c-1]A_0,  \ \ c\in \Z. 
\end{equation}
\noindent 

\begin{lem}
\label{lem:rationalpolynomial}
After extending scalars to $\Q(q)$ one has
\begin{equation}\label{eq_RQ_iso}
R_{\Q(q)}
:=
R\otimes_{\mcA}\Q(q)
\cong
\Q(q)[A_0,A_1]. 
\end{equation}
Furthermore,
\begin{equation}\label{eq_Kcn_prod}
K_{c,n}
=
\frac{1}{[n]!}
\prod_{r=c-n+1}^{c}
\bigl([r]A_1-[r-1]A_0\bigr)
\in R_{\Q(q)}, \  \ c\in\Z. 
\end{equation}
\end{lem}
For $n=1$, equation \eqref{eq_Kcn_prod} specializes to \eqref{eq_Kc1}, and, for $n=0$, to $K_{c,0}=1$. 

\begin{proof} Via~\eqref{eq_multiplication}, 
multiplication by $A_0,A_1$ gives the two degree-one multiplication formulae in $R$: 
\begin{align}
A_0 K_{c,n}
&=
[c+1]K_{c,n+1}
+
[n-c]K_{c+1,n+1},
\label{eq_PieriA0}\\
A_1 K_{c,n}
&=
[c]K_{c,n+1}
+
[n-c+1]K_{c+1,n+1}.
\label{eq_PieriA1}
\end{align}
These formulas hold for $0\le c\le n$ via \eqref{eq_multiplication}. To see that they hold for all integer $c$, let 
\[
F_c=[c+1]K_{c,n+1}+[n-c]K_{c+1,n+1}.
\]
Thus, \eqref{eq_PieriA0} asserts that $A_0K_{c,n}=F_c$.  The LHS  of \eqref{eq_PieriA0}, as a
sequence in $c$ of elements of $R_{n+1}$, satisfies \eqref{eq_q_binom}:
\[
\sum_{j=0}^{n+1}(-1)^j
\begin{bmatrix}n+1\\j\end{bmatrix}
A_0K_{c+j,n}=0.
\]
The sequence \(F_c\) satisfies the same recurrence.  Indeed, using
relation \eqref{eq_q_binom} for the degree $n+1$ elements \(K_{c,n+1}\), one has
\begin{equation*}
\sum_{j=0}^{n+1}(-1)^j
\begin{bmatrix}n+1\\j\end{bmatrix}F_{c+j}
\nonumber  =
[c+1]\sum_{j=0}^{n+2}(-1)^j
\begin{bmatrix}n+2\\j\end{bmatrix}
K_{c+j,n+1}
=0, 
\end{equation*}
where we used the identity
\begin{equation*}
\begin{bmatrix}n+1\\j\end{bmatrix}[c+j+1]
-
\begin{bmatrix}n+1\\j-1\end{bmatrix}[n-c-j+1]
=
[c+1]\begin{bmatrix}n+2\\j\end{bmatrix}.
\end{equation*}
Since $A_0K_{c,n}=F_c$ holds for $0\le c\le n$, the above argument implies that the equality holds for all $c\in \Z$. 
Equality \eqref{eq_PieriA1} can be established for all $c\in \Z$ in a similar fashion. 

\vspace{0.07in} 

We now use these formulas to prove the product expression for \(K_{c,n}\). Set 
\[
L_r=[r]A_1-[r-1]A_0=K_{r,1}.
\]
From \eqref{eq_PieriA0} and \eqref{eq_PieriA1}, now valid for arbitrary \(c\), we obtain
\begin{align*}
L_{c-n}K_{c,n}
={}&[c-n]A_1K_{c,n}-[c-n-1]A_0K_{c,n}\\
={}&
\bigl([c-n][c]-[c-n-1][c+1]\bigr)K_{c,n+1}+
\bigl([c-n][n-c+1]-[c-n-1][n-c]\bigr)
K_{c+1,n+1}.
\end{align*}
Coefficient at $K_{c+1,n+1}$ is zero, while
\[
[c-n][c]-[c-n-1][c+1]=[n+1].
\]
Therefore
\begin{equation}\label{eq_prod_KK}
K_{c-n,1}K_{c,n}=[n+1]K_{c,n+1},
\qquad c\in\Z.
\end{equation}
To check \eqref{eq_Kcn_prod} for any $c\in \Z$, use induction on $n$. 
For $n=1$, this is \eqref{eq_Kc1}. Assuming \eqref{eq_Kcn_prod} holds for $n$, 
relation \eqref{eq_prod_KK} gives
\begin{equation*}
K_{c,n+1}
=\frac{1}{[n+1]}K_{c-n,1}K_{c,n}
=\frac{1}{[n+1]!}
\prod_{r=c-n}^{c}
\bigl([r]A_1-[r-1]A_0\bigr),
\end{equation*}
which is the required formula for \(n+1\).

In particular, $R_{\Q(q)}$ is generated by its degree one elements $A_0,A_1$. Since  $R_n$ is a free $\mcA$-module of rank $n+1$, there can be no nontrivial relations on $A_0,A_1$, and the isomorphism \eqref{eq_RQ_iso} of commutative algebras holds. 
\end{proof}
Note that the natural homomorphism $R\lra R_{\Q(q)}$ is an inclusion. 

\vspace{0.07in} 

%%%%%%%%
% Lusztig form 
%%%%%%%%%

\subsection{Lusztig integral form \texorpdfstring{$\UzeroA$}{U(0,A)}} 

Consider the ring $\Q[q^{\pm 1},K^{\pm 1}]$ of Laurent polynomials in two commuting variables $q,K$ and localize it by inverting $q^n-q^{-n}$ for all $n>0$. Denote the resulting ring by $R'$. 
Let 
\begin{equation}\label{eq_Kc_one}
[K;c] := \frac{q^c K - q^{-c} K^{-1}}{q-q^{-1}} \in R'
\end{equation}
and define
\begin{equation}\label{eq_Kcn}
\bmat{K}{c}{n} := \frac{[K;c][K;c-1]\dots [K;c-n+1]}{[n]!}\in R'. 
\end{equation}
Here $c\in \Z$ is an integer and $n\ge 0$. 
We have 
\[
\bmat{K}{c}{0}=1, \ \forall c\in\Z, \ \mathrm{and} \ \mathrm{let} \  \bbinom{K}{n}:= \bmat{K}{0}{n}. 
\] 
Lusztig integral form $\UzeroA$ of the quantum Cartan subalgebra for $\slt$ is defined as the $\mcA$-subalgebra of $R'$ generated by $\bmat{K}{c}{n}$ over all $n\ge 0$, $c\in \Z$. This algebra contains the elements
\begin{equation}\label{eq_K_and_inverse}
K = \bmat{K}{1}{1} - q^{-1}\bbinom{K}{1} \  \mathrm{and} \ K^{-1}= \bmat{K}{1}{1} - q\bbinom{K}{1}. 
\end{equation}
A basis of $\UzeroA$ is given, for example, by 
\begin{equation}\label{eq_basis_U}
\Bigl\{ \bbinom{K}{n}: n\ge 0 \Bigr\} \cup \Bigl\{ \bmat{K}{1}{n}: n\ge 1 \Bigr\}, 
\end{equation}
see~\cite[Proposition~5.3]{GMSW1}. 
Recall from~\cite{KhCoherent}: 

\begin{prop}
There is a surjective $\mcA$-algebra homomorphism 
\begin{equation}\label{eq_psi}
\psi:R\lra \UzeroA
\end{equation}
given by 
\begin{equation}\label{eq_Kcn_int}
K_{c,n} \longmapsto 
\bmat{K}{c}{n}.
\end{equation}
\end{prop}
In particular, 
\begin{equation}\label{eq_Kc1_psi}
\psi(K_{c,1})= [K;c] = \frac{q^c K - q^{-c} K^{-1}}{q-q^{-1}}. 
\end{equation}
The homomorphism $\psi$ is not injective. Indeed, for $n\ge 2$ and $c\in \Z$ the element
\begin{equation}\label{eq_kernel_el}
K_{c+2,n}-(q^n+q^{-n})K_{c+1,n}+K_{c,n}-K_{c,n-2} 
\end{equation}
is nonzero in $R$ but its image in $\UzeroA$ is $0$, see \cite[Proposition 6.1]{GMSW1}.

%%%%%%%%%%%%%%
% Rees algebra
%%%%%%%%%%%%%%

\subsection{Rees algebra realization of \texorpdfstring{$R$}{R}} 

We now realize $R$ as a Rees algebra for a
two-step (parity) filtration on $\UzeroA$.
For brevity, let 
\begin{equation}\label{eq_Bcn}
B_{c,n}:=\bmat{K}{c}{n}\in\UzeroA,
\qquad c\in\Z,\  n\geq 0,
\end{equation}
and define
\begin{equation}\label{eq_Fn_def}
\mathscr F_n:=
\sum_{c\in\Z}\mcA\,B_{c,n}\subset\UzeroA,
\qquad n\geq 0.
\end{equation}
Also set $\mathscr F_{-1}=\mathscr F_{-2}=0$.

The relation of \cite[Proposition~6.1]{GMSW1} can be written as
\begin{equation}\label{eq_GMSW_lowering}
B_{c,n-2}
=
B_{c+2,n}-(q^n+q^{-n})B_{c+1,n}+B_{c,n},
\qquad n\geq 2.
\end{equation}
Consequently, 
\begin{equation}\label{eq_parity_include}
\mathscr F_{n-2}\subset \mathscr F_n,
\end{equation}
so that the even terms
$\mathscr F_0\subset\mathscr F_2\subset\mathscr F_4\subset\cdots$
and the odd terms
$\mathscr F_1\subset\mathscr F_3\subset\mathscr F_5\subset\cdots$
form two parity-separated increasing filtrations by $\mcA$-submodules of $\UzeroA$. Note 
that 
\[
\mathscr{F}_n\cap \mathscr{F}_m = \begin{cases} \mathscr{F}_{\min(n,m)} & \mathrm{if} \ n\equiv m(\mathrm{mod}\ 2), \\
0 & \mathrm{otherwise}, 
\end{cases}
\]

Basis \eqref{eq_basis_U} and relation \eqref{eq_GMSW_lowering} give a more
explicit description of these submodules.  Namely,
\begin{equation}\label{eq_Fn_basis}
\mathscr F_n
=
\bigoplus_{\substack{0\le j\le n\\ j\equiv n\!\!\!\pmod 2}}
\mcA\, B_{0,j}
\ \oplus\
\bigoplus_{\substack{1\le j\le n\\ j\equiv n\!\!\!\pmod 2}}
\mcA\, B_{1,j}.
\end{equation}
In particular, $\mathscr F_n$ is a free $\mcA$-module of rank $n+1$.
Indeed, for $n\geq 1$ one has
\begin{equation}\label{eq_Fn_recursive}
\mathscr F_n
=
\mathscr F_{n-2}
\oplus \mcA B_{0,n}
\oplus \mcA B_{1,n},
\end{equation}
while $\mathscr F_0=\mcA\cdot 1$.
The multiplication formula \eqref{eq_multiplication}, or equivalently
its counterpart in $\UzeroA$, implies
\begin{equation}\label{eq_F_mult}
\mathscr F_n\,\mathscr F_m\subset \mathscr F_{n+m}.
\end{equation}
Thus, \eqref{eq_parity_include} and \eqref{eq_F_mult} define a
multiplicative parity filtration on $\UzeroA$.  
Algebra $\UzeroA$ is $\Z/2$-graded, by the parity of $n$ in the spanning set $B_{c,n}$:
\begin{equation}\label{eq_U_Z2grading}
\UzeroA=\UzeroAzero\oplus\UzeroAone,
\end{equation}
with 
\begin{eqnarray}
    \UzeroAzero = \cup_{n\ge 0}\mathscr{F}_{2n}, \ \ 
    \UzeroAone = \cup_{n\ge 0}\mathscr{F}_{2n+1}. 
\end{eqnarray}
Let $t$ be an indeterminate and define the corresponding
\emph{parity Rees algebra}
\begin{equation}\label{eq_parity_Rees}
\operatorname{Rees}^{(2)}_{\mathscr F}
:=
\bigoplus_{n\geq 0} t^n\mathscr F_n
\ \subset\ \UzeroA[t].
\end{equation}
The superscript $(2)$ records that the filtration is increasing
in steps of two.  If one reserves the term ``Rees algebra'' for an
ordinary filtration $\mathscr F_n\subset\mathscr F_{n+1}$, then
\eqref{eq_parity_Rees} may instead be called \emph{the parity Rees algebra}.  Notice
that $t\notin \Rees$, whereas $t^2\in \Rees$. 

\begin{prop}\label{prop_R_Rees}
There is a natural isomorphism of graded $\mcA$-algebras
\begin{equation}\label{eq_R_Rees}
\Theta:
R\xrightarrow{\ \cong\ }
\Rees,
\qquad
K_{c,n}\longmapsto t^n\bmat{K}{c}{n}.
\end{equation}
\end{prop}

\begin{proof}
The relation \eqref{eq_q_binom} is satisfied by the elements
$B_{c,n}$, and the multiplication relation \eqref{eq_multiplication}
is exactly the corresponding multiplication formula for Lusztig's
elements.  Hence \eqref{eq_R_Rees} defines a homomorphism of graded
$\mcA$-algebras.

By the definition of $\mathscr F_n$, the degree $n$ component $\Theta_n:R_n\longrightarrow t^n\mathscr F_n$  is surjective.  The source is a free $\mcA$-module of rank $n+1$, with
basis $K_{0,n},\dots,K_{n,n}$, while the target is a free $\mcA$-module
of rank $n+1$ by \eqref{eq_Fn_basis}.  Therefore $\Theta_n$ is an
isomorphism for every $n$, and hence so is $\Theta$.
\end{proof}

We next describe the Rees parameter intrinsically inside $R$.  Set
\begin{equation}\label{eq_kappas}
\kappa_+:=A_1-q^{-1}A_0=K_{1,1}-q^{-1}K_{0,1},
\qquad
\kappa_-:=A_1-qA_0=K_{1,1}-q K_{0,1},
\qquad
\kappa:=\kappa_+\kappa_-.
\end{equation}
Then 
\begin{equation}\label{eq_kappa_expanded}
\kappa
=
A_1^2-(q+q^{-1})A_0A_1+A_0^2=K_{0,2}-(q^2+q^{-2})K_{1,2}+ K_{2,2}. 
\end{equation}
Notice that this equality is the $n=2,c=0$ case of the homogeneous GMSW relation 
\[
K_{c+2,n}-(q^n+q^{-n})K_{c+1,n}+K_{c,n} = \kappa K_{c,n-2}. 
\]

Using \eqref{eq_K_and_inverse}, or directly from the definitions, we
obtain
\begin{equation}\label{eq_theta_kappas}
\Theta(\kappa_+)=tK, \qquad
\Theta(\kappa_-)=tK^{-1}, \qquad
\Theta(\kappa)=t^2.
\end{equation}
In particular,
\begin{equation}\label{eq_psi_kappas}
\psi(\kappa_+)=K, \qquad
\psi(\kappa_-)=K^{-1}, \qquad \psi(\kappa)=1.
\end{equation}
The first two equalities in \eqref{eq_psi_kappas} are given by 
\cite[Remark~5.4]{GMSW1}.

The map $\psi$ now has a simple interpretation.  Evaluation at $t=1$
defines a surjective homomorphism
\begin{equation}\label{eq_eval_one}
\operatorname{ev}_{t=1}:
\Rees \longrightarrow \UzeroA, \qquad t^n f\longmapsto f,
\end{equation}
and
\begin{equation}
    \psi \ = \ \operatorname{ev}_{t=1} \circ \Theta, \ \ \psi: R\lra \UzeroA.  
\end{equation}

\begin{cor}\label{cor_kernel_psi}
The kernel of $\psi$ is the principal ideal generated by
$\kappa-1$.  Consequently,
\begin{equation}\label{eq_U_as_quotient}
\UzeroA\cong R/(\kappa-1).
\end{equation}
\end{cor}

\begin{proof}
In view of the isomorphism $\Theta$, it is enough to prove that the kernel of
\eqref{eq_eval_one} is generated by $t^2-1$.  Split an element in this
kernel into its even and odd $t$-degree parts.  Since
$\UzeroA=\mcU_{\mcA}^{\overline 0}\oplus\mcU_{\mcA}^{\overline 1}$, the two
parts separately vanish at $t=1$.
For either parity, the coefficients lie in a nested chain
\[
\mathscr F_\epsilon\subset
\mathscr F_{\epsilon+2}\subset
\mathscr F_{\epsilon+4}\subset\cdots.
\]
Hence the usual division by $t^2-1$ stays inside the parity Rees
algebra: explicitly, the coefficients of the quotient are partial
sums of the coefficients of the original element, and the inclusions
above place these partial sums in the required $\mathscr F_n$.
Therefore
\[
\ker(\operatorname{ev}_{t=1})=(t^2-1).
\]
Now use \eqref{eq_theta_kappas}.
\end{proof}

The relation of \cite[Proposition~6.1]{GMSW1} has a homogeneous lift
in $R$.  Namely, for $n\geq 2$,
\begin{equation}\label{eq_homogeneous_GMSW}
K_{c+2,n}-(q^n+q^{-n})K_{c+1,n}+K_{c,n}
=
\kappa K_{c,n-2}.
\end{equation}
Indeed, after applying $\Theta$, equation
\eqref{eq_homogeneous_GMSW} is precisely
$t^n$ times \eqref{eq_GMSW_lowering}.  Specializing $\kappa$ to $1$
recovers the relation in $\UzeroA$.  In particular, the previously
displayed element of $\ker(\psi)$ is
\begin{equation*}
K_{c+2,n}-(q^n+q^{-n})K_{c+1,n}
   +K_{c,n}-K_{c,n-2} =(\kappa-1)K_{c,n-2}.
\end{equation*}

Finally, the rational polynomial model from
Lemma~\ref{lem:rationalpolynomial} is exactly the same Rees
construction after extending scalars to $\Q(q)$.  Formally write
$z=q^x$.  Then
\begin{equation}\label{eq_ReesAA}
A_0\longmapsto t[x],
\qquad
A_1\longmapsto t[x+1],
\end{equation}
and more generally
\begin{equation}\label{eq_ReesKcn}
K_{c,n}\longmapsto
t^n\bbinom{x+c}{n}.
\end{equation}
Moreover,
\begin{equation}\label{eq_Reeskappa}
\kappa_+\longmapsto tz,
\qquad
\kappa_-\longmapsto tz^{-1},
\qquad
\kappa\longmapsto t^2.
\end{equation}
After tensoring with $\Q(q)$ one has
\[
\mathscr F_n\otimes_{\mcA}\Q(q)
=
\bigoplus_{j=0}^{n}\Q(q)\,K^{\,n-2j},
\]
and therefore
\[
\Rees \otimes_{\mcA}\Q(q) \cong \Q(q)[tK,tK^{-1}],
\]
which is a polynomial algebra in the two algebraically independent
elements $tK$ and $tK^{-1}$.  This recovers  the isomorphism \eqref{eq_RQ_iso}.

 The quotient description \eqref{eq_U_as_quotient}
also explains the residual $\Z/2$-grading on Lusztig's integral form \eqref{eq_U_Z2grading} versus the $\Z$-grading on $R$, since  
$\deg\kappa=2$. The relation $\kappa=1$ merges the degrees of $R$ that differ by $2$. 

%%%%%%%%%%%%%%%%%%%%%%
%
% Equivariant coherent sheaves 
%
%%%%%%%%%%%%%%%%%%%%%%

\section{A colimit category for equivariant coherent sheaves on \texorpdfstring{$\PP^{n}$}{P(n)}}
\label{sec_colimit_cat}

\subsection{The balanced torus action and the stabilizing divisor}

{\bf Equivariant sheaves, immersions $i_n$ and skyscraper sheaves.}
Let $\kk$ be an algebraically closed field, and let
$T=\mbGm$ act on
\begin{equation} \label{eq_V}
V=\kk u_+\oplus\kk u_-
\end{equation}
with weights $+1$ and $-1$. We identify $
\PP^1=\PP(V)$, $[a:b]=\kk(a u_++b u_-)$, 
and 
\[
\PP^n=\PP(\Sym^n V) \cong \Sym^n(\PP^1).
\]
We use the equivariant linearization of $\mcO_{\PP^n}(1)$ for which
$H^0(\PP^n,\mcO(1))$ has character $
[n+1]=q^n+q^{n-2}+\cdots+q^{-n}.$

Recall from~\cite{KhCoherent} the natural identification of $\mcA$-modules 
\begin{equation}\label{eq_isoRn_PPn}
R_n \cong K^T(\PP^n) 
\end{equation}
between the weight $n$ homogeneous component of $R$ and the Grothendieck group of the category of $T$-equivariant coherent sheaves on $\PP^n$. Under this isomorphism element $K_{c,n}\in R_n$ is mapped to the class $[\mcO(-c)]$ of the line bundle $\mcO(-c)$ on $\PP^n$. 
The symmetrization (or addition) maps 
\begin{equation}\label{eq_fnm}
f_{n,m}: \PP^n\times \PP^m \lra \PP^{n+m}, \ \ (D,D') \longmapsto D+D' 
\end{equation}
are finite, flat maps of smooth algebraic varieties and they induce exact functors $(f_{n,m})_{\ast}$ and $f_{n,m}^{\ast}$ between categories of $T$-equivariant coherent sheaves: 
\begin{equation}\label{eq_maps_for_fnm}
\begin{tikzcd}[cramped, sep=large]
\Coh^T(\PP^n\times \PP^m) 
  \arrow[r, shift left=0.7ex, "{(f_{n,m})_*}"] 
& \Coh^T(\PP^{n+m}) 
  \arrow[l, shift left=0.7ex, "{(f_{n,m})^*}"]
\end{tikzcd}
\end{equation}
Summing over all $n$ and over all $n,m$ we define 
\[
\PP:=\sqcup_{n\ge 0} \PP^n, \ \  f: \PP\times\PP \lra \PP,  \ \  f=\sqcup_{n,m\ge 0}f_{n,m},
\]
viewing $\PP$ as a scheme which is a disjoint union of countably many smooth varieties. 
Form the direct sum of categories 
\[
\Coh^T(\PP) \ := \ \oplusop{n\ge 0} \Coh^T(\PP^n).
\]
This is an abelian symmetric monoidal category with the biexact tensor product (convolution) 
\begin{equation}\label{eq_convolution}
\mcF_1 \star \mcF_2 := (f_{n,m})_{\ast}(\mcF_1\boxtimes \mcF_2), 
\end{equation}
for equivariant coherent sheaves $\mcF_1,\mcF_2$ on $\PP^n$ and $\PP^m$, respectively. 
 This tensor bifunctor gives maps of $\mcA$-modules 
 \[
 K^T(\PP^n)\otimes K^T(\PP^m) \stackrel{[(f_{n,m})_{\ast}]}{\lra}
 K^T(\PP^{n+m})
 \]
 which correspond to the multiplication map on homogeneous components  $R_n\otimes_{\mcA} R_m\lra R_{n+m}$ 
under the isomorphism \eqref{eq_isoRn_PPn}. 
Summing over all $n$ and over all $n,m$ (for maps) gives an isomorphism of $\mcA$-algebras 
\begin{equation}\label{eq_iso_R_PP}
R \cong K^T(\PP) = \oplusop{n\ge 0}K^T(\PP^n),
\end{equation}
see~\cite{KhCoherent} for details. 

\vspace{0.07in} 

Let $ p_+=[1:0]=\PP(\kk u_+),\ \ p_-=[0:1]=\PP(\kk u_-)$
be the two fixed points of the action of $T$ on $\PP^1$. If $x_+,x_-$ are the dual homogeneous coordinates, then $\mathrm{wt}(x_+)=-1$, $\mathrm{wt}(x_-)=1$. 
The point $p_+$ is the zero locus of $x_-$, whose weight is $1$, giving us a short exact sequence 
\begin{equation}
\label{eq_fixedplus}
0\lra q\,\mcO_{\PP^1}(-1)
\stackrel{x_-}{\lra} \mcO_{\PP^1}
\lra \mcO_{p_+} \lra 0
\end{equation}
of $T$-equivariant sheaves on $\PP^1$. 
Here, $q^a\mcL$, for a $T$-equivariant sheaf $\mcL$, denotes the twist of $\mcL$ by the corresponding one-dimensional representation of $T$. 
Likewise, there is a short exact sequence 
\begin{equation}
\label{eq_fixedminus}
0\lra q^{-1}\,\mcO_{\PP^1}(-1)
\stackrel{x_+}{\lra} \mcO_{\PP^1}
\lra \mcO_{p_-} \lra 0
\end{equation}
for the point $p_-$. 
Write
\begin{equation}\label{eq_D0}
D_0=p_++p_-\in\PP^2.
\end{equation}
Addition of $D_0$ gives the closed immersion
\begin{equation}\label{eq_stabilize}
i_n:\PP^n\hookrightarrow\PP^{n+2}, 
\qquad D\longmapsto D+D_0.
\end{equation}

Under the isomorphism~\eqref{eq_isoRn_PPn}, elements $K_{0,1}$ and $K_{1,1}$ of $R_1$ map to the classes of bundles $\mcO_{\PP^1}$ and $\mcO_{\PP^1}(-1)$ (respectively) in the Grothendieck group. 
From equations~\eqref{eq_fixedplus}, \eqref{eq_fixedminus} and~\eqref{eq_kappas} we get 
\begin{eqnarray*}
[\mcO_{p_+}] & = & [\mcO_{\PP^1}]- q [\mcO_{\PP^1}(-1)]= K_{0,1}-q K_{1,1} = -q\kappa_+, \\   
 {[\mcO_{p_-}]} & = & 
[\mcO_{\PP^1}]- q^{-1} [\mcO_{\PP^1}(-1)]= K_{0,1}-q^{-1} K_{1,1} =-q^{-1}\kappa_-. 
\end{eqnarray*}
We have 
\[
\mcO_{p_+}\star\mcO_{p_-}\cong \mcO_{D_0},
\]
and therefore, using \eqref{eq_fixedplus}, \eqref{eq_fixedminus},
and \eqref{eq_kappas},
\[
[\mcO_{D_0}]
=
[\mcO_{p_+}][\mcO_{p_-}]
=
(-q\kappa_+)(-q^{-1}\kappa_-)
=
\kappa.
\]
Thus the class of the structure sheaf of the stabilizing divisor is
the element $\kappa\in R_2$.

\vspace{0.07in}

\noindent 
{\bf Categorical counterpart of the homogeneous GMSW relation~\eqref{eq_homogeneous_GMSW}.}
For every \(n\geq0\), let
\begin{equation}
\label{eq_iota_abelian}
\iota_n:=(i_n)_*:
\Coh^T(\PP^n)\lra\Coh^T(\PP^{n+2}).
\end{equation}
Since $i_n$ is a closed immersion, $\iota_n$ is an exact and fully
faithful functor between abelian categories.  Moreover, the restriction of
\[
f_{2,n}:\PP^2\times\PP^n\lra\PP^{n+2}
\]
to \(\{D_0\}\times\PP^n\) is \(i_n\), so there is a natural isomorphism
of exact functors
\begin{equation}\label{eq_iota_convolution}
\iota_n(\mcF)
\cong
\mcO_{D_0}\star\mcF.
\end{equation}
Consequently, on equivariant Grothendieck groups,
\begin{equation}\label{eq_iota_kappa_abelian}
[\iota_n(\mcF)]=\kappa[\mcF],
\end{equation}
so that the exact functor $\iota_n$ corresponds to the multiplication by $\kappa$ on the Grothendieck group via the isomorphism 
\eqref{eq_iso_R_PP}. 

Multiplication by
$u_+u_-$ identifies $\Sym^nV$ with the subspace of
$\Sym^{n+2}V$ spanned by the non-extreme weight vectors.  Hence
$i_n(\PP^n)\subset\PP^{n+2}$ is the common zero locus of the two
extreme-weight sections of $\mcO_{\PP^{n+2}}(1)$.  Denote these
sections by $z_+$ and $z_-$, with
\[
\operatorname{wt}(z_+)=n+2,
\qquad
\operatorname{wt}(z_-)=-n-2.
\]
They form a regular sequence.  Therefore, for every $c\in\Z$, the
equivariant Koszul resolution of
$(i_n)_*\mcO_{\PP^n}(-c)$ is
\begin{equation}\label{eq_Koszul_iota_abelian}
\begin{split}
0\lra {}&
\mcO_{\PP^{n+2}}(-c-2)
\stackrel{d_1}{\lra}
\left(
q^{n+2}\mcO_{\PP^{n+2}}(-c-1)
\oplus
q^{-(n+2)}\mcO_{\PP^{n+2}}(-c-1)
\right)
\\
&\stackrel{d_2}{\lra}
\mcO_{\PP^{n+2}}(-c)
\lra
(i_n)_*\mcO_{\PP^n}(-c)
\lra 0,
\end{split}
\end{equation}
where
\[
d_1=\begin{pmatrix} -z_- \\ z_+\end{pmatrix}
\qquad
d_2=
(z_+ , z_-).
\]
Passing to the Grothendieck group gives
\[
K_{c+2,n+2}
-(q^{n+2}+q^{-n-2})K_{c+1,n+2}
+K_{c,n+2}
=
\kappa K_{c,n},
\]
which is exactly \eqref{eq_homogeneous_GMSW} after replacing $n+2$
by $n$.  Thus, the homogeneous GMSW relation~\eqref{eq_homogeneous_GMSW} is  categorified
by the Koszul exact sequence \eqref{eq_Koszul_iota_abelian}.

%%%%%%%%%%%%%%%%%%%%%
% Abelian colimit category 
%%%%%%%%%%%%%%%%%%%%%

\subsection{The abelian colimit of coherent sheaves categories}

For $\epsilon\in\{0,1\}$ consider the sequence of abelian categories and functors 
\begin{equation}
\Coh^T(\PP^\epsilon)
\xrightarrow{(i_\epsilon)_*}
\Coh^T(\PP^{\epsilon+2})
\xrightarrow{(i_{\epsilon+2})_*}
\cdots.
\label{eq:abeliantelescope}
\end{equation}
Each functor $(i_{\epsilon+2m})_*$ is exact and fully faithful.

\begin{definition}\label{eq_def_colimit_cat}
Let
\[
\Coh^{\epsilon,T}_{\PP} =
\colim_r \Coh^T(\PP^{\epsilon+2r})
\]
be the colimit category for the directed system \eqref{eq:abeliantelescope}, and set 
\[
\Coh^T_{\PP} =
\Coh^{0,T}_{\PP}\oplus
\Coh^{1,T}_{\PP}.
\]
\end{definition}
For $n\geq 0$, let $\bar n\in\mathbb Z/2$ denote the parity of $n$.
We write
\begin{equation}\label{eq_jmath_n}
\jmath_n:
\Coh^T(\PP^n)
\lra
\Coh_{\PP}^{\bar{n},T}
\end{equation}
for the canonical functor from the category of equivariant coherent sheaves on $\PP^n$  to the 
colimit category and 
\begin{equation}\label{eq_jmath}
\jmath: \Coh^T(\PP) \lra \Coh_{\PP}^T, \ \ \jmath = \sum_{n\ge 0}\jmath_n 
\end{equation}
for the sum of these functors over all $n$.

This colimit category can be treated very concretely: its objects are (equivariant) coherent
sheaves on finite-dimensional projective spaces $\PP^{\epsilon+2m}$, where the morphisms between two objects are computed as follows. 
Assume that we are given (equivariant coherent) sheaves $\mcF_1$ and $\mcF_2$ on $\PP^{\epsilon+2m_1}$ and $\PP^{\epsilon+2m_2}$, respectively. 
There are corresponding objects $\jmath(\mcF_1)$ and $\jmath(\mcF_2)$ in $\Coh^T_{\PP}$. 
Assume, for example, that $m_1\le m_2$. Then 
\begin{equation}\label{eq_hom_stabilize}
\Hom_{\Coh^T_{\PP}}(\jmath(\mcF_1),\jmath(\mcF_2))\cong \Hom_{\Coh^T(\PP^{\epsilon+2 m_2})}(\widetilde{i}(\mcF_1),\mcF_2), 
\end{equation}
where $\widetilde{i}$ is the (equivariant) pushforward functor for the composition of the inclusions 
\[
\PP^{\epsilon+2m_1}\subset \PP^{\epsilon+2(m_1+1)}\subset \dots \subset \PP^{\epsilon+2m_2}. 
\]
The hom spaces are computed likewise if $m_1\ge m_2$. 

\begin{prop}
\label{prop_abeliancolimit}
The category $\Coh^T_{\PP}$ is abelian and the functors $\jmath_n$  above are exact and fully faithful. Functor $\jmath$ is exact and faithful. Hom spaces in $\Coh^T_{\PP}$ are finite-dimensional $\kk$-vector spaces, and they can be computed at finite stages as described in \eqref{eq_hom_stabilize}. 
\end{prop}

\begin{proof}
A morphism in the colimit category can be represented at a finite stage (on $\PP^{\epsilon+2m}$ for some $\epsilon$ and $m$) as explained above.
Because every $(i_n)_*$ is exact and fully faithful, kernels and
cokernels computed at that stage remain kernels and cokernels after
every further transition.  Consequently, kernels and cokernels are
well-defined in the filtered colimit, and the abelian category axioms hold in $\Coh^T_{\PP}$.  
Since all transition functors $\iota_n$ are exact and fully faithful,
the functors $\jmath_n$ are exact and fully faithful as well. 
Note that the functor $\jmath$ is exact and faithful but not full, since passing to the colimit leads to possible nonzero morphisms between the images of sheaves on $\PP^n$ and $\PP^m$ for $n\not= m$ of the same parity. 
Finally, by \eqref{eq_hom_stabilize}, every Hom space in
$\Coh^T_{\PP}$ is a Hom space between equivariant coherent sheaves
on a finite-dimensional projective space, thus it is finite-dimensional.
\end{proof}
 The above functors 
satisfy canonical natural isomorphisms
\begin{equation}\label{eq_jmath_stabilization}
\jmath_{n+2}\circ\iota_n
\cong
\jmath_n.
\end{equation}
In particular, a coherent sheaf and any of its further stabilizations define isomorphic objects of $\Coh^T_{\PP}$.
The functors $\jmath_n$ are compatible with character twists:
$\jmath_n(q^a\mcF)\cong q^a\jmath_n(\mcF)$.

\begin{remark}
Passing to the colimit respects all finite-stage categorical
relations: the functors $\jmath_n$ preserve and reflect isomorphisms,
and preserve exact sequences and direct-sum decompositions.  The new
identifications introduced in the colimit are precisely those coming
from the stabilization isomorphisms~\eqref{eq_jmath_stabilization}. 
\end{remark}

Recall that the maps $f_{n,m}:\PP^n\times\PP^m\longrightarrow\PP^{n+m}$ are  finite and flat~\cite[Proposition~3.1]{KhCoherent}, and the convolution 
\begin{equation}\label{conv_finite}
\mcF_1\star \mcF_2
:=
(f_{n,m})_*(\mcF_1\boxtimes \mcF_2)
\end{equation}
is exact in each variable on coherent sheaves
\cite[Proposition~3.2]{KhCoherent}.
The addition maps satisfy
\[
f_{n+2,m}\circ(i_n\times\Id)
=
i_{n+m}\circ f_{n,m},
\]
and likewise for the second variable.
Therefore
\begin{equation}
(i_n)_*\mcF_1\star \mcF_2
\cong
(i_{n+m})_*(\mcF_1\star \mcF_2),
\qquad
\mcF_1\star(i_m)_*\mcF_2
\cong
(i_{n+m})_*(\mcF_1\star \mcF_2).
\label{eq_multstabilization}
\end{equation}

\begin{prop}
\label{prop_abelianmonoidal}
Convolution extends to a biexact symmetric monoidal bifunctor on
$\Coh^T_{\PP}$:
\begin{equation}\label{eq_star_PP}
\Coh^T_{\PP}\times \Coh^T_{\PP} \stackrel{\star}{\lra} \Coh^T_{\PP}.
\end{equation}
Its unit is the image of
$\mcO_{\PP^0}$.
\end{prop}

\begin{proof}
Compatibility with the transition maps is
\eqref{eq_multstabilization}. 
Associativity and symmetry are inherited from addition of effective
divisors. Biexactness holds for the tensor product at finite stages~\eqref{eq_convolution}, see also~\cite{KhCoherent}, and immediately extends to the biexactness of the convolution on the colimit category.  
\end{proof}
In particular, for $\mcF_1\in\Coh^T(\PP^n)$ and
$\mcF_2\in\Coh^T(\PP^m)$ there is a canonical isomorphism
\begin{equation}\label{eq_jmath_convolution}
\jmath_n(\mcF_1)\star\jmath_m(\mcF_2)
\cong
\jmath_{n+m}(\mcF_1\star\mcF_2).
\end{equation}

\noindent 
{\bf The Grothendieck ring.}
\begin{lem}
\label{lem:gradedcolimit}
Let $S=\bigoplus_{n\geq0}S_n$ be a graded ring and let
$\eta\in S_d$ be central and homogeneous.
Then
\[
\bigoplus_{\epsilon\in\Z/d}
\colim_r
\left(
S_{\epsilon+rd}
\xrightarrow{\cdot\eta}
S_{\epsilon+(r+1)d}
\right)
\cong
S/(\eta-1)
\]
as a $\Z/d$-graded ring.
\end{lem}

\begin{proof}
The direct limit is the $\Z/d$-graded version of the homogeneous
localization $S[\eta^{-1}]$.  Since imposing $\eta=1$ already makes
$\eta$ invertible,
\[
S/(\eta-1)
\cong
S[\eta^{-1}]/(\eta-1),
\]
which gives the result.  Equivalently, the isomorphism follows
directly from the universal property of the filtered colimit.
\end{proof}

The $\mcA$-linear map on equivariant Grothendieck groups induced by
$(i_n)_*$ is multiplication by
\[
[\mcO_{D_0}]=\kappa.
\]
Since taking the Grothendieck group $K_0$ commutes with filtered unions of abelian categories of the
form above, Lemma~\ref{lem:gradedcolimit} and equation~\eqref{eq_U_as_quotient} give the following result.

\begin{thm}[Abelian categorification of the integral Cartan algebra]
\label{thm_abelianK0}
There is an isomorphism of $\Z/2$-graded $\mcA$-algebras
\begin{equation}
\label{eq_abelianmain}
K_0(\Coh^T_{\PP})\cong
\mcU^0_{\mcA},
\end{equation}
under which the class represented by
$\mcO_{\PP^n}(-c)$ maps to
\begin{equation}\label{eq_Bcn_two}
B_{c,n}=\begin{bmatrix}K;c\\n\end{bmatrix}.
\end{equation}
The monoidal product in
$\Coh^T_{\PP}$ categorifies multiplication in
$\mcU^0_{\mcA}$.
\end{thm}

For $c\in\mathbb Z$ we define
\begin{equation}\label{eq_Bscript_definition}
\mathscr B_{c,n}
:=
\jmath_n\bigl(\mcO_{\PP^n}(-c)\bigr)
\in
\Coh^{\bar{n},T}_{\PP}.
\end{equation}
Under the identification~\eqref{eq_abelianmain}, the class
$[\mathscr B_{c,n}]$ corresponds to
$B_{c,n}$ in \eqref{eq_Bcn_two}, and the object $\mathscr B_{c,n}$ of $\Coh^T_{\PP}$ can be thought of as a categorification of $B_{c,n}$. 
Since the functors $\jmath_n$ are exact and fully faithful, all
isomorphisms, exact sequences, direct-sum decompositions, and
equivariant character twists which hold at a finite stage (in $\Coh^T(\PP^n)$, for some $n$) remain valid
after passing to the colimit.  In particular, the splitting formulas
for the pushforwards under $f_{n,m}$ give corresponding decompositions
in $\Coh^T_{\PP}$.

In particular, let $0\leq c\leq n, 0\leq d\leq m$. The inequalities $-m\leq c-d\leq n
$ hold automatically, so the equivariant splitting formula in~\cite{KhCoherent} for
$(f_{n,m})_*$ applies to
$\mcO_{\PP^n}(-c)\boxtimes\mcO_{\PP^m}(-d).$
Consequently,
\begin{equation}\label{eq_Bscript_product}
\mathscr B_{c,n}\star\mathscr B_{d,m}
\cong
\bigoplus_k
\bbinom{n-c+d}{k-c}
\bbinom{m-d+c}{k-d}\,
\mathscr B_{k,n+m},
\end{equation}
where $\max(c,d)\leq k\leq\min(n+d,m+c)$.
(Recall our notation, where for 
$P(q)=\sum_a p_aq^a\in\mathbb N[q,q^{-1}]$ 
we write $P(q)\mathscr B
:=
\bigoplus_a(q^a\mathscr B)^{\oplus p_a}$). 

Decomposition~\eqref{eq_Bscript_product} categorifies the Guti\'errez-Mart\'inez-Szwej-Wildon relation in $\UzeroA$, see~\cite[Theorem~1.2]{GMSW1}. This is also the relation \eqref{eq_multiplication}, with $K_{a,b}$'s replaced by $B_{a,b}$'s (and see~\cite{KhCoherent} for its categorification in the lift $R$ of $\UzeroA$).

\begin{remark}
Functor $(i_n)_*$ in \eqref{eq_iota_abelian} preserves the hom spaces between the objects but it does not preserve the Ext groups. The latter groups grow large as the functor is iterated to embed objects from $\Coh^T(\PP^n)$ into $\Coh^T(\PP^{n+2N})$ for large $N$. Although these Ext groups are bigraded, by cohomological degree and
$T$-weight, even individual bidegree components can become
infinite-dimensional in the direct limit.  It is nevertheless natural
to expect that a suitable renormalized Euler characteristic of the
finite-stage Ext groups, or equivalently of their stable limit, leads
to a $q$-semilinear form related to $\UzeroA$.  We do not study this
question in the present paper.   
\end{remark}

%%%%%%%%%%%%%%%%%%%%%%
%
%  Cat of comultiplication
%
%%%%%%%%%%%%%%%%%%%%%%

\section{A categorification of the  comultiplication in \texorpdfstring{$\UzeroA$}{U(0,A)}}
\label{sec_abelianDelta}

{\bf The product category.}
For two parities $\epsilon,\epsilon'\in\Z/2$, form the filtered
colimit
\[
\Coh^{\epsilon,\epsilon'}_{\PP\times \PP}
=
\colim_{r,s}
\Coh^T
\left(
\PP^{2r+\epsilon}\times
\PP^{2s+\epsilon'}
\right),
\]
where the transition maps add $D_0$ independently in the first and
second factors.  The torus acts diagonally on each product.
Set
\[
\Coh^T_{\PP\times\PP}
=
\bigoplus_{\epsilon,\epsilon'}
\Coh^{\epsilon,\epsilon'}_{\PP\times \PP}.
\]

Componentwise divisor addition gives the 
abelian category $\Coh^T_{\PP\times\PP}$ a biexact symmetric monoidal structure.
Equivariant K\"unneth gives
\begin{equation}
K_0(\Coh^T_{\PP\times\PP})
\cong
K_0(\Coh^T_{\PP})
\otimes_{\mcA}
K_0(\Coh^T_{\PP})\cong \UzeroA\otimes_{\mcA}\UzeroA.
\label{eq_Kunnethlimit}
\end{equation}

\noindent 
{\bf The diagonal pushforward.}
Let $ \delta_n:\PP^n\hookrightarrow\PP^n\times\PP^n$
be the diagonal.

\begin{definition}
For $\mcF\in\Coh^T(\PP^n)$ set
\begin{equation}
\Delta^\heartsuit_n(\mcF)
=
(\delta_n)_*(\mcF(1)).
\label{eq:Dheart}
\end{equation}
\end{definition}

The functor is exact, since tensoring by $\mcO(1)$ and closed-immersion
pushforward are exact.
The square
\[
\begin{tikzcd}
\PP^n \arrow[r,"\delta_n"] \arrow[d,"i_n"'] &
\PP^n\times\PP^n
\arrow[d,"i_n\times i_n"]\\
\PP^{n+2} \arrow[r,"\delta_{n+2}"] &
\PP^{n+2}\times\PP^{n+2}
\end{tikzcd}
\]
commutes, and $i_n^*\mcO(1)\cong\mcO(1)$.
Hence
\[
\Delta^\heartsuit_{n+2}(i_{n*}\mcF)
\cong
(i_n\times i_n)_*
\Delta^\heartsuit_n(\mcF),
\]
and equation \eqref{eq:Dheart} descends to the colimit.

\begin{thm}
\label{thm_abelianDelta}
There is an exact functor
\[
\Delta^\heartsuit: \Coh^T_{\PP}
\lra \Coh^T_{\PP\times \PP}
\]
which is symmetric monoidal, coassociative and cocommutative.
\end{thm}

\begin{proof}
Let $ \mcF_1\in\Coh^T(\PP^n)$,
$\mcF_2\in\Coh^T(\PP^m)$.
The identity
\[
\delta_{n+m}\circ f_{n,m}
=
(f_{n,m}\times f_{n,m})
\circ(\delta_n\times\delta_m)
\]
gives
\begin{align*}
\Delta^\heartsuit_n(\mcF_1)
\star
\Delta^\heartsuit_m(\mcF_2)
&\cong
(\delta_{n+m})_*
(f_{n,m})_*
\bigl(\mcF_1(1)\boxtimes \mcF_2(1)\bigr)\\
&\cong
(\delta_{n+m})_*
\bigl((\mcF_1\star \mcF_2)(1)\bigr),
\end{align*}
where the second isomorphism uses $
f_{n,m}^*\mcO(1)\cong\mcO(1,1)$ 
and the projection formula.
This proves monoidality.

Iterating the construction gives natural isomorphisms between the sheaves 
\[
(\Delta^\heartsuit\boxtimes\Id)\Delta^\heartsuit(\mcF)\cong (\Id\boxtimes\Delta^\heartsuit)\Delta^\heartsuit(\mcF)\cong 
(\delta_n^{(3)})_*(\mcF(2)),
\]
where $\delta_n^{(3)}:\PP^n\to(\PP^n)^3$ is the small diagonal.
This gives coassociativity.
Cocommutativity follows from invariance of the diagonal under the
transposition of the two factors.
\end{proof}

The exact functor $\Delta^\heartsuit$ induces an $\mcA$-linear map
\begin{equation}\label{eq_Delta_C}
\Delta_{\mathcal C}
:=
K_0(\Delta^\heartsuit):
K_0(\Coh^T_{\PP})
\longrightarrow
K_0(\Coh^T_{\PP\times\PP})
\cong
K_0(\Coh^T_{\PP})\otimes_{\mcA}K_0(\Coh^T_{\PP}).
\end{equation}
Since $\Delta^\heartsuit$ is monoidal, $\Delta_{\mathcal C}$ is an
algebra homomorphism. Let $\Delta_{\mcU}$ denote the standard Cartan coproduct 
\begin{equation}\label{eq_Delta_U}
\Delta_{\mcU} \ : \ \UzeroA \lra \UzeroA\otimes_{\mcA} \UzeroA, \ \ 
\Delta_{\mcU}(K) = K\otimes K. 
\end{equation}

\vspace{0.07in}

\noindent 
{\bf Comparison with the standard Cartan coproduct.}
Let
\[
\Psi:
K_0(\Coh^T_{\PP})
\xrightarrow{\sim}
\UzeroA 
\]
denote the isomorphism in
Theorem~\ref{thm_abelianK0}, and $\Psi\otimes \Psi$ denote the one in \eqref{eq_Kunnethlimit}.
Define the parity-twisted isomorphism
\begin{equation}
\Phi([\mcF])
=
(-1)^n\Psi([\mcF])
\qquad
\text{if }\mcF\text{ is represented on }\PP^n.
\label{eq:Phi}
\end{equation}
This is well-defined because stabilization changes $n$ by two, and
it is a ring isomorphism because degrees add under convolution. Recall our notation $B_{c,n}$ from  \eqref{eq_Bcn_two} and 
\[ A_0 = K_{0,1}=[\mcO_{\PP^1}], \ \ 
A_1 = K_{1,1}=[\mcO_{\PP^1}(-1)] 
\]
 from~\eqref{eq_A01}. We continue to denote by $A_0,A_1$ their images in
$K_0(\Coh^T_{\PP})$ under $K_0(\jmath_1)$.  We have 
\[
B_{0,1}=\Psi(A_0),\qquad B_{1,1}=\Psi(A_1).
\]

\begin{prop}
\label{prop_degreeoneDelta}
We have the following identities for $\PP^1$: 
\begin{align}
(\Psi\otimes\Psi)[\Delta^\heartsuit_1(\mcO)]
&= (\Psi\otimes\Psi)\Delta_{\mcC}(A_0)=[2]\,B_{0,1}\otimes B_{0,1} -B_{1,1}\otimes B_{0,1}
-B_{0,1}\otimes B_{1,1},
\label{eq_DA0}
\\
(\Psi\otimes\Psi)[\Delta^\heartsuit_1(\mcO(-1))]
&=(\Psi\otimes\Psi)\Delta_{\mcC}(A_1)=
B_{0,1}\otimes B_{0,1}-B_{1,1}\otimes B_{1,1}.
\label{eq_DA1}
\end{align}
These are the negatives of the standard
Cartan coproducts of $B_{0,1}$ and $B_{1,1}$.
\end{prop}

\begin{proof}
The diagonal in $\PP^1\times\PP^1$ has the equivariant divisor
resolution
\[
0\lra \mcO(-1,-1) \lra
\mcO \lra \mcO_\Delta \lra 0.
\]
Hence
\[
[\Delta^\heartsuit_1(\mcO(-1))]
=[\delta_*\mcO]=
A_0\otimes A_0-A_1\otimes A_1.
\]
Tensoring the resolution by $\mcO(1,0)$ and using
$ [\mcO_{\PP^1}(1)]=[2]A_0-A_1$
gives
\[
[\Delta^\heartsuit_1(\mcO)]
=
[2]A_0\otimes A_0
-A_1\otimes A_0
-A_0\otimes A_1.
\]
Applying $\Psi\otimes \Psi$ gives \eqref{eq_DA1} and \eqref{eq_DA0}, respectively. 
On the other hand,  
\[
K=B_{1,1}-q^{-1}B_{0,1},
\qquad
K^{-1}=B_{1,1}-qB_{0,1}.
\]
Therefore, the standard comultiplication $\Delta_{\mcU}$ in $\UzeroA$ satisfies 
\begin{align*}
\Delta_{\mcU}(B_{0,1})
&=
K\otimes B_{0,1}
+
B_{0,1}\otimes K^{-1}\\
&=
B_{1,1}\otimes B_{0,1}+B_{0,1}\otimes B_{1,1}-[2]B_{0,1}\otimes B_{0,1},
\end{align*}
which is the negative of \eqref{eq_DA0}.
Similarly,
\[
\Delta_{\mcU}(B_{1,1})
=
B_{1,1}\otimes B_{1,1}-B_{0,1}\otimes B_{0,1},
\]
the negative of \eqref{eq_DA1}.
\end{proof}
Recall that $\Delta_{\mcC}=K_0(\Delta^{\heartsuit})=[\Delta^{\heartsuit}]$, see \eqref{eq_Delta_C}. 

\begin{thm}[Abelian categorification of the coproduct]
\label{thm:coproductK0}
Under the parity-twisted isomorphism \eqref{eq:Phi},
the exact functor $\Delta^\heartsuit$ categorifies the standard
Cartan coproduct:
\begin{equation}
(\Phi\otimes\Phi)\,
\Delta_{\mcC}
=
\Delta_{\mcU}\,
\Phi.
\label{eq:coproductintertwine}
\end{equation}
In particular, one recovers the formula 
\begin{equation}
\Delta_{\mcU}
\begin{bmatrix}K;c\\n\end{bmatrix}
=
\sum_{r+s=n}
q^{cs}
K^s
\begin{bmatrix}K;c\\r\end{bmatrix}
\otimes
K^{-r}
\begin{bmatrix}K\\s\end{bmatrix}.
\label{eq:CartanDelta}
\end{equation}
\end{thm}

\begin{proof}
Both $K_0(\Delta^\heartsuit)$ and the standard Cartan coproduct $\Delta_{\mcU}$ are
ring homomorphisms.
By Lemma~\ref{lem:rationalpolynomial}, after extension to $\Q(q)$
the algebra $R$, and hence its quotient
\[
K_0(\Coh^T_{\PP})\cong R/(\kappa-1),
\]
is generated by the images of the degree one elements $A_0,A_1$.
Proposition~\ref{prop_degreeoneDelta} shows that on each degree one
generator the geometric map is the negative of the standard coproduct.
Hence, if $x\in R_n$ is homogeneous and $\bar{x}$ denotes its image in
$K_0(\Coh^T_{\PP})$, then
\[
(\Psi\otimes\Psi)
K_0(\Delta^\heartsuit)(\bar{x})
=
(-1)^n
\Delta_{\mcU}(\Psi(\bar{x})).
\]
The target is $\mcA$-torsion-free, so the equality holds integrally.
Inserting the parity factors from \eqref{eq:Phi} gives
\eqref{eq:coproductintertwine}.
Formula \eqref{eq:CartanDelta} follows from
\[
\Delta_{\mcU}([K;c])
=
q^cK\otimes[K;0]+[K;c]\otimes K^{-1}
\]
and the balanced $q$-Vandermonde identity; see
\cite[Chapter~3]{LusztigBook} for the Hopf structure on the quantum group
and compare \cite[Corollary~2.12]{GMSW1} for $q$-Vandermonde.
\end{proof}

Note that $\Delta^{\heartsuit}$ is a functor, $\Delta_{\mcC}$ its decategorification, and $\Delta_{\mcU}$ the standard coproduct on $\UzeroA$ being categorified. 

\begin{remark}
The abelian construction above gives exact multiplication and
comultiplication.  The natural counit is no longer an exact functor
between abelian categories, but it has a simple derived form.  At a
finite stage $n$ set
\[
\varepsilon_n^{\mathrm{der}}(\mcF)
=
R\Gamma(\PP^n,\mcF(-1)).
\]
The projection formula and $i_n^*\mcO(-1)\cong\mcO(-1)$ give canonical
isomorphisms
\[
\varepsilon_{n+2}^{\mathrm{der}}((i_n)_*\mcF)
\cong
\varepsilon_n^{\mathrm{der}}(\mcF),
\]
so these functors are compatible with stabilization.  Moreover,
\[
(\varepsilon^{\mathrm{der}}\boxtimes\id)\Delta^\heartsuit
\cong\id
\cong
(\id\boxtimes\varepsilon^{\mathrm{der}})\Delta^\heartsuit.
\]
Consequently, the growth of higher Ext groups under the embeddings $i_n$ does
not obstruct the derived counit.  
For the distinguished objects $\mathscr B_{c,n}
=\jmath_n(\mcO_{\PP^n}(-c))$
we have
\begin{equation}\label{eq_counit_Bcn}
\varepsilon^{\mathrm{der}}(\mathscr B_{c,n})
=
R\Gamma\bigl(\PP^n,\mcO_{\PP^n}(-c-1)\bigr).
\end{equation}
In particular, in the range $0\leq c\leq n$,
\begin{equation}\label{eq_counit_Bcn_range}
\varepsilon^{\mathrm{der}}(\mathscr B_{c,n})
\cong
\begin{cases}
0, & 0\leq c<n,\\[1mm]
\kk[-n], & c=n.
\end{cases}
\end{equation}
Here $\kk$ carries the trivial $T$-action.
More generally, for any $c\in\mathbb Z$,
\begin{equation}\label{eq_counit_Bcn_general}
\varepsilon^{\mathrm{der}}(\mathscr B_{c,n})
\cong
\begin{cases}
H^0\bigl(\PP^n,\mcO_{\PP^n}(-c-1)\bigr),
& c\leq -1,\\[1mm]
0,
& 0\leq c<n,\\[1mm]
H^0\bigl(\PP^n,\mcO_{\PP^n}(c-n)\bigr)^\vee[-n],
& c\geq n.
\end{cases}
\end{equation}
For the standard counit
\[
\varepsilon_{\mathcal U}:\UzeroA\longrightarrow \mcA,
\qquad
\varepsilon_{\mathcal U}(K)=1,
\]
one has, for all $c\in\mathbb Z$ and $n\geq 0$,
\begin{equation}\label{eq_counit_U_Bcn}
\varepsilon_{\mathcal U}(B_{c,n})
=\varepsilon_{\mathcal U}\left(
\begin{bmatrix}K;c\\n\end{bmatrix}
\right)
=\bbinom{c}{n},
\end{equation}
\end{remark}
using the standard extension of balanced $q$-binomials to all $c\in \Z$. 

%%%%%%%%%%%%%%%%%
%
% Nonequivariant sheaves
%
%%%%%%%%%%%%%%%%%

\section{Nonequivariant sheaves and a categorification of the ring of integer-valued polynomials}
\label{sec_noneq}

There is a classical, nonequivariant shadow of our construction which
connects it with the ring of integer-valued polynomials~\cite{CC} and its spectrum (the Hilbert additive group scheme~\cite{HM}).  After forgetting
the $T$-equivariance, the convolution algebra
\[
\oplusop{n\geq 0} K_0(\Coh(\PP^n))
\]
identifies with the Rees algebra of the degree filtration on the ring
\[
\Int(\Z)
=
\{f(x)\in\Q[x]\mid f(\Z)\subseteq\Z\}
\]
of integer-valued polynomials, via
\[
[\mcO_{\PP^n}(-c)]
\longmapsto
t^n\binom{x+c}{n}.
\]
The Hopf algebra $\Int(\Z)$ is the coordinate ring of the Hilbert
additive group scheme, and its degree filtration and Rees algebra
are studied, in particular, in \cite{HM}.  Thus, the
nonequivariant decategorification of the construction in the present paper and in~\cite{KhCoherent} recovers the Hopf algebra of integer-valued polynomials and its Rees algebra.  What is new
in the present setting is the realization of this structure by the
abelian categories $\Coh(\PP^n)$, convolution along the divisor-addition
maps, and their stabilized colimit; moreover, the $T$-equivariant
construction gives a $q$-deformation whose stabilized
Grothendieck ring is Lusztig's integral Cartan form
$\UzeroA$. 

\vspace{0.07in}

We now briefly describe the nonequivariant specialization of the
construction and its relation with $\Int(\Z)$.
Using the notation from~\cite{KhCoherent}, let 
\[
\wtR =
\bigoplus_{n\geq0}K_0\bigl(\Coh(\PP^n)\bigr)
\]
be the direct sum of Grothendieck groups of categories of (nonequivariant) coherent sheaves on $\PP^n$, over all $n$. 
The multiplication on $\wtR$ is induced, as before, by the pushforwards for the finite divisor-addition maps  $f_{n,m}$. 
Writing $\wtK_{c,n}=[\mcO_{\PP^n}(-c)]$, 
the multiplication formula is obtained from the equivariant one by
specializing $q=1$:
\[
\wtK_{c,n}\wtK_{b,m}
=
\sum_k
\binom{n-c+b}{k-c}
\binom{m-b+c}{k-b}
\wtK_{k,n+m}.
\]
Recall the ring of integer-valued polynomials 
\[
\Int(\Z) = \{f(x)\in\Q[x]\mid f(\Z)\subseteq\Z\},
\]
and let $F_n\Int(\Z)$ be the subgroup of elements of degree at most
$n$.  Form the Rees algebra: 
\[
\operatorname{Rees}_{\deg}\Int(\Z) =
\bigoplus_{n\geq0} t^nF_n\Int(\Z) \subseteq
\Int(\Z)[t].
\]
There is then a graded algebra isomorphism
\begin{equation}\label{eq_q1-Rees-identification}
\wtR
\overset{\cong}{\longrightarrow}
\operatorname{Rees}_{\deg}\Int(\Z),
\qquad \wtK_{c,n} \longmapsto t^n\binom{x+c}{n}.
\end{equation}
Indeed, after setting $X=tx$ one obtains
\[
t^n\binom{x+c}{n} =
\frac{(X+ct)(X+(c-1)t)\cdots(X+(c-n+1)t)}{n!}.
\]
For $c=0$ this is
\[
\frac{X(X-t)\cdots(X-(n-1)t)}{n!},
\]
the standard generator of the Rees deformation of the degree
filtration on $\Int(\Z)$.  This filtration and its Rees algebra
appear in the study of the Hilbert additive group scheme~\cite{HM}.

There are two natural stabilizations in the nonequivariant setting.
First, after choosing a point $p\in\PP^1$, one may use the one-step
embeddings
\[
i_n\colon\PP^n\hookrightarrow\PP^{n+1},
\qquad
D\longmapsto D+p.
\]
The corresponding transition functor is $(i_n)_*$, and on
Grothendieck groups the transition is multiplication by $[\mcO_p] =\wtK_{0,1}-\wtK_{1,1}$. 
Under \eqref{eq_q1-Rees-identification},
\[
\wtK_{0,1}\longmapsto tx, \qquad
\wtK_{1,1}\longmapsto t(x+1),
\]
hence
\[
[\mcO_p]\longmapsto -t.
\]
Since multiplication by $[\mcO_p]$ becomes the identity in the
colimit, the one-step stabilization imposes $t=-1$.  Therefore
\begin{equation}\label{eq_q1-one-step}
K_0\left(
\colim_n\Coh(\PP^n)
\right)
\cong
\Int(\Z).
\end{equation}
With the normalization
\[
[\jmath_n\mcO_{\PP^n}(-c)]
\longmapsto
(-1)^n\binom{x+c}{n},
\]
the coproduct induced by the diagonal functor is the standard
binomial coproduct
\begin{equation}\label{eq:q1-binomial-coproduct}
\Delta\binom{x}{n}
=
\sum_{r+s=n}
\binom{x}{r}\otimes\binom{x}{s}.
\end{equation}
Equivalently, $x$ is primitive:
\[
\Delta(x)=x\otimes1+1\otimes x,
\qquad
\epsilon(x)=0,
\qquad
S(x)=-x.
\]
Consequently, the symmetric monoidal category 
\[
\Coh(\PP)\ :=\ \oplusop{n\ge 0} \Coh(\PP^n),
\]
with the tensor product given by the exact functor $f_{\ast}$, categorifies the Rees algebra $\operatorname{Rees}_{\deg}\Int(\Z)$.

The one-step nonequivariant colimit category $\Coh_{\PP}=\colim_n\Coh(\PP^n)$ (nonequivariant one-step $\PP^n\subset \PP^{n+1}\subset \dots$ analogue of the category in Definition~\ref{eq_def_colimit_cat}) categorifies the bialgebra of integer-valued polynomials $\Int(\Z)$ (which is also the coordinate ring of the Hilbert additive group scheme
$\mathbb H=\mathrm{Spec}(\Int(\Z))$), with $K_0(\Coh_{\PP})\cong\Int(\Z)$.

\begin{remark}
A different, semisimple, categorification of $\Int(\Z)$ was found by Harman, Snowden, and Snyder in~\cite{HSS}. That categorification is given by the \emph{positive Delannoy category}, the subcategory of their Delannoy category where the simple objects are restricted to $L_{\bullet^n}$, over all $n\ge 0$. The positive Delannoy category is also discussed in the forthcoming paper~\cite{KSDel}. 
\end{remark}

One may instead keep the two-step stabilization used in the
equivariant construction.  Choosing two points $p_+,p_-\in\PP^1$, let
\[
D_0=p_++p_-,
\qquad
i_n(D)=D+D_0.
\]
The transition is multiplication by
\[
[\mcO_{D_0}]
=
[\mcO_{p_+}]\star[\mcO_{p_-}],
\]
which corresponds under \eqref{eq_q1-Rees-identification} to $t^2$.
Hence the two-step colimit imposes
\[
t^2=1.
\]
The resulting Hopf algebra is
\begin{equation}\label{eq:q1-two-step}
\Int(\Z)\otimes_{\Z}\Z[g]/(g^2-1),
\end{equation}
where $g$ records the residual parity.
Its bialgebra structure is then categorified by the nonequivariant version of the category $\Coh^T_{\PP}$ and the functors described in Sections~\ref{sec_colimit_cat} and~\ref{sec_abelianDelta}. 

On the categorical level, the group-like element $g$ is replaced by the pair of objects 
\[
\mathscr G_+
=
\jmath_1(\mcO_{p_+}),
\qquad
\mathscr G_-
=
\jmath_1(\mcO_{p_-})
\]
in the odd summand of the two-step colimit category. Since
\[
\mathscr G_+\star\mathscr G_-
\cong
\jmath_2(\mcO_{p_++p_-})
\cong\mathbf 1,
\]
the objects $\mathscr G_+$ and $\mathscr G_-$ are mutually inverse
with respect to convolution.  Consequently,  convolution with either
one exchanges the even and odd summands of the colimit category. Moreover,
$[\mathscr G_+]=[\mathscr G_-]$
in the nonequivariant Grothendieck group, since the classes of all points of $\PP^1$ coincide.  Under the parity-twisted identification
with \eqref{eq:q1-two-step}, their common class is $g$.  Hence, 
$g^2=1$ is categorified by  $\mathscr G_+$ and
$\mathscr G_-$ being mutually inverse objects with the same
Grothendieck class.  Furthermore,
\[
\Delta^\heartsuit(\mathscr G_\pm)
\cong
\mathscr G_\pm\boxtimes\mathscr G_\pm,
\]
categorifying the group-like identity
$\Delta(g)=g\otimes g$.

\begin{remark}\label{rem_Moonen-Polishchuk}
The use of symmetric powers together with convolution by divisor
addition has a close antecedent in the work of Moonen--Polishchuk
\cite{MP}.  For a family of smooth curves
$C\to S$, they consider the relative symmetric powers
$C^{[n]}=\Sym_S^n(C)$ 
and equip the sum of relative Chow homology groups 
\[
\bigoplus_{n\geq 0}\mathrm{CH}_*(C^{[n]}/S)
\]
with the Pontryagin product defined by proper pushforward along the
addition maps
\[
C^{[n]}\times_S C^{[m]}
\longrightarrow C^{[n+m]},
\qquad
(D,D')\longmapsto D+D'.
\]
The geometric correspondence underlying Moonen-Polishchuk's multiplication is
the same one that underlies the convolution product in the present
paper.  Our construction may be viewed as a coherent sheaf
categorical refinement of this general symmetric power convolution
picture in the special case $C=\PP^1$: instead of applying Chow
homology to the symmetric powers, we use the categories
$\Coh^T(\PP^n)$ themselves and the exact pushforward functors along
the addition maps.  The resulting Grothendieck ring and its balanced
stabilization are consequently different from the Chow-theoretic
Pontryagin ring studied in \cite{MP}; in particular,
the $T$-equivariant structure is what leads here to Lusztig's
integral Cartan algebra.
\end{remark}

\printbibliography
%%%%%%%%%%%%%%%%%%%%%%%%%%%%%%%%%
%\bibliographystyle{amsplain}
%\bibliography{biblio}

\vspace{2em} \noindent
Department of Mathematics, Johns Hopkins University, Baltimore, MD 21218, USA\\
\textit{E-mail address:} \href{mailto:khovanov@jhu.edu}{khovanov@jhu.edu}

\end{document}